\documentclass[12pt]{article}
\usepackage{amscd,amssymb,amsmath,verbatim,amsthm, graphicx}
\usepackage{color}
\usepackage{hyperref}

\newcommand{\bS}{\mathbb S}

\newcommand{\Hol}{\operatorname{Hol}}

\newcommand{\CC}{\mathbb C}

\newcommand{\NN}{\mathbb N}

\newcommand{\RR}{\mathbb R}
\newcommand{\ZZ}{\mathbb Z}
\newcommand{\del}{\partial}

\newcommand{\Ric}{\mathrm{Ric}}
\newcommand{\Diff}{\mathrm{Diff}}

\newcommand{\calC}{{\mathcal C}}
\newcommand{\calD}{{\mathcal D}}

\newcommand{\calO}{{\mathcal O}}

\newcommand{\calV}{{\mathcal V}}

\newcommand{\frakf}{\mathfrak f}

\newcommand{\frakr}{\mathfrak r}

\newcommand{\Fr}{\mathrm{Fr}}
\newcommand{\Spec}{\mathrm{Spec}}
\newcommand{\Vol}{\mathrm{Vol}}

\newtheorem{theorem}{Theorem}
\newtheorem{theoremIntro}{Theorem}

\newtheorem*{theorem*}{Theorem}
\newtheorem{proposition}{Proposition}
\newtheorem{corollary}{Corollary}
\newtheorem{lemma}{Lemma}
\theoremstyle{definition}
\newtheorem{definition}{Definition}
\theoremstyle{remark}
\newtheorem*{remark}{Remark}
\newtheorem{ex}{Example}

\title{Spectral properties of reducible conical metrics}

\author{Bin Xu
\and Xuwen Zhu}

\newcommand{\Addresses}{{
  \bigskip
  \footnotesize

 \noindent Bin Xu\\ \textsc{CAS Wu Wen-Tsun Key Laboratory of Mathematics and School of Mathematical Sciences, University of Science and Technology of China, Hefei 230026, Anhui, China}\par\nopagebreak
  \textit{E-mail address}: \texttt{bxu@ustc.edu.cn}

  \medskip

 \noindent Xuwen Zhu\\ \textsc{Department of Mathematics, Northeastern University, Boston 02115, Massachusetts, USA}\par\nopagebreak
  \textit{E-mail address}: \texttt{x.zhu@northeastern.edu}
}}

\date{}

\begin{document}
\maketitle

\begin{abstract}
We show that the monodromy of a spherical conical metric $g$ is reducible if and only if the metric $g$ has a real-valued eigenfunction with eigenvalue~2 for the holomorphic extension $\Delta_g^{\rm Hol}$ of the associated Laplace--Beltrami operator. Such an eigenfunction produces a meromorphic vector field, which is then related to the developing maps of the conical metric. We also give a lower bound of the first nonzero eigenvalue of
$\Delta_g^{\rm Hol}$,
together with a complete classification of the dimension of the space of real-valued $2$-eigenfunctions for $\Delta_g^{\rm Hol}$ depending on the monodromy of the metric $g$. This paper can be seen as a new connection between  the complex analysis method and the PDE approach in the study of spherical conical metrics.
\end{abstract}


\section{Introduction}

The study of the interplay between the geometry and the spectrum of geometrically related operators has a long history and has produced a lot of interesting results. In this paper, we study how the monodromy of a spherical conical metric influences the spectrum of the associated Laplace--Beltrami operator. Our results connect two areas of research which attracted considerable attention in recent years. One is the study of metrics of constant positive curvature with conical singularities on compact Riemann surfaces (which we call ``spherical conical metrics''), and another is the spectral theory of the Laplace--Beltrami operators on singular surfaces. For the theory of spherical conical metrics, reducible monodromy is expected to occur at the singular points of the moduli space of such metrics. For the theory of Laplace operators, metrics with conical singularities give an interesting class of examples and exhibit many surprising phenomena. This paper appears to be the first to establish a connection between those two areas.

We start by introducing the basic setup.
Let $\Sigma$ be a compact Riemann surface, ${\frak p}=(P_1,\cdots, P_n)$ be an $n$-tuple of distinct points on $\Sigma$, and $\vec \beta=(\beta_1,\cdots,\beta_n)\in (\RR_{+}\setminus \{1\})^{n}$ be an $n$-dimensional vector.  We say $g$ is a {\em conical metric} representing the divisor $D=\sum_{j=1}^n\,(\beta_j-1)[P_j]$ on $\Sigma$, if
$g$ is a smooth conformal metric on the punctured surface $\Sigma\backslash {\rm supp}\, D= \Sigma\backslash \{P_1,\cdots, P_n\}$ and has conical singularities of angle $2\pi\beta_j$ at $P_j$ for $j=1, \dots, n$. The latter condition means that near $P_j$ there is a complex coordinate $z$ such that $z(P_{j})=0$ and $g$ can be written as $g=e^{2u}|dz|^2$ where $u-(\beta_j-1)\log\,|z|$ extends to $z=0$ continuously.
We say a conical metric $g$ representing $D$ is {\em spherical} if
$g$ has constant curvature one on the punctured surface $\Sigma\backslash {\rm supp}\, D$.

A classical way to view spherical conical metrics is through their developing maps. For any such metric $g$, there exists a multi-valued locally univalent meromorphic function
$f: \Sigma\backslash {\rm supp}\, D \to \overline{\Bbb C}$,
called a {\it developing map} of $g$, such that $g$ is given by the pullback by $f$ of the standard spherical metric. Such a developing map has the following three properties (cf. \cite[Lemma 2.1 and Lemma 3.1]{CWWX2015}):
\begin{enumerate}
\item (Pull-back) Denote the standard metric on the sphere by $g_{\rm st}=\frac{4|dw|^2}{(1+|w|^2)^2}$ for $w\in \overline{\CC}$, then $g=f^{*}g_{\rm st}$  on $\Sigma\backslash {\rm supp}\, D$;
\item (Monodromy) The monodromy of $f$ is contained in
$${\rm PSU}(2)=\left\{w\mapsto\frac{aw+b}{-\overline{b}w+\overline{a}}:\ a,\ b\in \mathbb{C},\ \vert a \vert^2+\vert b \vert^2=1\right\};$$
\item (Cone angle) Near each $P_j$, the principal singular term of the Schwarzian derivative of $f$ is given by $\frac{1-\beta_j^2}{2z^2}$.
\end{enumerate}
We note here that for a given spherical conical metric, its developing map is not unique, and all such maps are related by M\"obius transforms in ${\rm PSU}(2)$. So for a given metric, the monodromy of all its developing maps are contained in the same conjugacy class of ${\rm PSU}(2)$.
In this paper we are in particular interested in the following class of metrics:
\begin{definition}
A spherical conical metric $g$ is called {\it reducible} if there exists for the metric $g$ a developing map with monodromy in  ${\rm U}(1)=\{w\mapsto e^{i\theta}w:\theta\in [0,\, 2\pi)\}$. Such a developing map of $g$ is called {\it multiplicative}. The metric $g$ is called {\it trivially reducible} if the monodromy of its developing map is trivial.
\end{definition}

One can also view a spherical conical metric as a solution to the following singular Liouville equation:
\begin{equation}\label{e:liou}
\Delta_{g_{0}}u-e^{2u}+K_{g_{0}}=0
\end{equation}
where $g_{0}$ is a conical metric with the prescribed conical singularities but not necessarily with constant curvature one, and $g=e^{2u}g_{0}$ gives the sought-after  spherical conical metric in the same conformal class. Here $\Delta_{g_{0}}$ is the associated Laplace--Beltrami operator of $g_{0}$. When some of the cone angles are bigger than $2\pi$, the existence and uniqueness of solutions to~\eqref{e:liou} is still not completely understood. One approach is via perturbation near a given spherical conical metric $g$, which leads one to study the linearized operator of the above equation, given by
\[
\Delta_{g}-2.
\]
Unlike for
complete metrics, in order for $\Delta_{g}$ of a conical metric to be self-adjoint, boundary conditions are needed. One common choice is the Friedrichs extension $\Delta_{g}^{\Fr}$, which is the extension such that the domain only consists of bounded functions on $\Sigma$. This is also the extension one uses to solve the perturbation problem of~\eqref{e:liou}.
It is known that when all cone angles are less than $2\pi$, the first nonzero eigenvalue $\lambda_{1}$ of $\Delta_{g}^{\Fr}$ is bounded below by $2$, and $\lambda_{1}=2$ if and only if $g$ is a spherical football~\cite{LuTi92, MaWe}. However, when some of the cone angles are bigger than $2\pi$, $2$ is no longer the lower bound, and the deformation is obstructed exactly when $2$ is in the spectrum of $\Delta_{g}^{\Fr}$. In~\cite{MaZh19} it is shown that the deformation can be unobstructed by ``splitting'' cone points, and there is a trichotomy of deformation rigidity depending on the dimension of eigenspace with eigenvalue $2$.
In addition to the Friedrichs extension, in this paper we also consider another extension called the holomorphic extension $\Delta_{g}^{\Hol}$, which was introduced in~\cite{Hill} in the case of flat conical metrics. We show that the functions in the domain of this extension are also closely related to the spectral geometry of reducible conical metrics.

Now we state the main result of this paper, which is a spectral characterization of spherical conical metrics with reducible monodromy. Denote by $\Delta_{g}$ the Laplace--Beltrami operator of a spherical conical metric $g$, and $\calD^{\Hol}$ (resp. $\calD^{\Fr}$) the domain of the holomorphic (resp. Friedrichs) extension of $\Delta_{g}$. The main theorem is stated below (for a more precise statement see Theorem~\ref{t:eigen} and Theorem~\ref{t:reducible}):
\begin{theoremIntro}
A spherical conical metric $g$ has reducible monodromy if and only if there is a real-valued eigenfunction $\phi\in \calD^{\Hol}$ satisfying
$$
\Delta_{g} \phi=2\phi.
$$
\end{theoremIntro}

There has been a lot of recent development in understanding spherical conical metrics. One of the features of this problem is that it can be approached from many aspects of mathematics including complex analysis, min-max theory, integrable systems, synthetic geometry, etc., see~\cite{McOw, Tro, LuTi92, Ere, UY,  BDM, Car, CLW, LW, CKL, MP, MP2, Ere2017, Ere19, EG, Ka, Dey, CWWX2015, SCLX, MaWe, MaZh17, MaZh19, Zh18} and the references therein.   This paper can be seen as a new connection between  the complex analysis method and the PDE approach.

The study of reducible conical metrics was initiated in~\cite{UY}, and has seen a lot of development recently~\cite{CWWX2015, SongXu2020, Ere2017}. One feature of reducible metrics is that there exist multiplicative developing maps of such a metric which give meromorphic differentials (sometimes called ``character one-forms'') that are dual to meromorphic vector fields (\cite{CWWX2015}, also see Section~\ref{ss:oneform}). One expects that there will be constraints on the divisor $D=\sum_{j=1}^n\,(\beta_j-1)[P_j]$ for reducible metrics. 
When $\Sigma$ is the Riemann sphere, Song and the first author~\cite{SongXu2020} determined
the angle constraints when all the angles are in $2\pi{\Bbb Q}$. Recently Eremenko~\cite{Ere2017} gave a complete answer on the angle constraint problem on the Riemann sphere.
There is ongoing work of the first author and his collaborators~\cite{CLSX} on the case when the genus of $\Sigma$ is positive. In~\cite{Zh19} the local rigidity of one family of such metrics was shown by using synthetic geometry which exemplifies the constraints on ${\rm supp}\, D$.

The number 2 also appears as the upper bound for the first nonzero eigenvalue in the eigenvalue isoperimetric problem among all smooth metrics on $\bS^{2}$, where the standard spherical metric is the only extremal metric for $\lambda_{1}$~\cite{Her, KNPP}.
There are also two analogues in K\"ahler geometry. One is that a positive lower bound on Ricci curvature gives a lower bound on the first nonzero eigenvalue~\cite{Li, Ob}; the other is that the complex gradient vector fields of 2-eigenfunction of a K\"ahler-Einstein metric on a Fano manifold
are holomorphic and form a reductive Lie algebra \cite{Ma57}.
Our proof can be seen as an analogue of~\cite{Li, Ob, Ma57} where a similar Bochner technique is used to obtain a lower bound for the first nonzero eigenvalue and, in the case of equality,  produce meromorphic vector fields from 2-eigenfunctions (see Section~\ref{ss:merovec}).
Here we prove the following result (see Theorem~\ref{t:lambda} and Proposition~\ref{p:hol}):
\begin{theoremIntro}
For any spherical conical metric $g$, the first nonzero eigenvalue of $\Delta_{g}^{\Hol}$ satisfies $\lambda_{1}\geq 2$. If there is a 2-eigenfunction $\phi \in \calD^{\Hol}$, then the complex gradient vector field $X:=\phi^{,z}\partial_{z}$ is meromorphic on $\Sigma$.
\end{theoremIntro}
The existence of such meromorphic vector fields indicate some symmetry of the metric itself. These vector fields can also be viewed as generators of gauge transformations, which are obstructions in solving the nonlinear uniformization problem (see examples and discussion in Section~\ref{s:sp}).


If in addition the 2-eigenfunction is real-valued, we then show that the metric is reducible by relating the meromorphic vector field to a developing map.
As an application we are also able to show that the dimension of the space of real-valued 2-eigenfunctions is completely determined by the monodromy (see Theorem~\ref{thm:hol_eigenspace}):
\begin{theoremIntro}
For a reducible spherical metric $g$, the dimension of real-valued 2-eigenfunctions of $\Delta_{g}^{\Hol}$ is either 1 or 3. The dimension equals to 3 if and only if $g$ is trivially reducible.
\end{theoremIntro}

We point out here that the real-valued 2-eigenfunctions we find are actually in the domains of both the Friedrichs extension and the holomorphic extension. Therefore all reducible metrics are included in the obstructed case as discussed in~\cite{MaZh19}. However, having $2$ in the spectrum of the Friedrichs Laplacian does {\em not} imply that the monodromy is reducible, and in fact there is evidence that there exist irreducible metrics with~2 in their Friedrichs spectrum.
In addition the assumption in Theorem A that the eigenfunction is real-valued is  also essential. See more discussion in Section~\ref{s:discussion}.


We also mention another type of metrics called HCMU metrics. The Gaussian curvature functions of these metrics behave similarly to the real-valued eigenfunctions discussed above. There is also a corresponding existence of meromorphic vector fields and character 1-forms. See~\cite{Che00, LiZhu02, CCW, CW, CWX} and the references therein for details.

This paper is organized as follows. In Section~\ref{s:sa} we describe various self-adjoint extensions of the Laplace--Beltrami operator of a spherical conical metric. In Section~\ref{s:eigen} we construct appropriate eigenfunctions assuming the metric is reducible. In Section~\ref{s:sp} we prove a lower bound for eigenvalues,  and in the case of equality we prove the reducible monodromy property by producing a meromorphic vector field from a real-valued 2-eigenfunction which is then related to a developing map. In Section~\ref{s:discussion} we discuss the relation of our work to existing works and open problems.

\section{Self-adjoint extensions of the operator $\Delta_{g}$}\label{s:sa}

Consider the Laplace--Beltrami operator of a spherical conical metric $g$, denoted by $\Delta_{g}$, acting on $\calC_{c}^{\infty}(\Sigma\backslash {\rm supp}\, D)$.  Locally near a cone point of angle $2\pi\beta$, the metric is given by the geodesic coordinates as
$$g=d\frakr^{2}+\beta^{2}\sin^{2} \frakr d\theta^{2} , \ (\frakr, \theta) \in (0,\epsilon)\times \RR/2\pi\ZZ.$$
Under these coordinates the Laplace--Beltrami operator is locally given by
$$
\Delta_{g}=-\partial_{\frakr}^{2} - \frac{\cos \frakr}{\sin \frakr}\partial_{\frakr} -\frac{1}{\beta^{2}}\frac{1}{\sin^{2} \frakr} \partial_{\theta}^{2}.
$$

This operator is of conical type, which has been extensively studied~\cite{Cheeger, BS1, BS2, Mooers}. Conical operators can be viewed as rescaled versions of b-operators~\cite{Me}.
We now briefly recall some notation here. Let $\Sigma_{D}:=[\Sigma; {\rm supp}\, D]$ be the surface $\Sigma$ with cone points blown up, that is, each puncture is replaced by a circle and
polar coordinates are introduced near the cone points. Denote by $\calV_{b}$ the b-vector fields on $\Sigma_{D}$, which are smooth in the interior and locally given by a basis of $\{\frakr\partial_{\frakr}, \partial_{\theta}\}$ near the punctures. Let $\Diff^{m}_{b}(\Sigma_{D})$ be the space of b-differential operators of order not exceeding $m$, locally of the form
\[
A = \sum_{j + \ell \leq m} a_{j \ell}(\frakr,\theta) (\frakr\del_\frakr)^j \del_\theta^\ell,\, a_{j\ell} \in \calC^\infty(\Sigma_{D}).
\]
A conical operator is a rescaled version of a b-operator, given by elements in $\frakr^{-m}\Diff^{m}_{b}(\Sigma_{D})$. In particular, the Laplace operator $\Delta_{g}$ can be written as $-\frakr^{-2}[(\frakr\partial_{\frakr})^{2}+\beta^{-2}\partial_{\theta}^{2}]+\dots$ where the remaining terms are smooth multiples of $\frakr^2\del_\frakr$ and $\frakr \del_\theta$ hence lower order, therefore $\Delta_{g}\in \frakr^{-2}\Diff^{2}_{b}(\Sigma_{D})$. Let $L_{b}^{2}(\Sigma_{D})$ be the $L^{2}$ space with respect to the b-measure which is locally given by $\frac{d\frakr}{\frakr}\otimes d\theta$. Note that it is related to the $L^{2}$ space associated with $d\Vol_{g}$ by the following relation: $\frakr^{-1}L_{b}^{2}(\Sigma_{D})=L^{2}(\Sigma_{D}, d\Vol_{g})$.
We also denote by $H_{b}^{\ell}(\Sigma_{D})$ the b-Sobolev space with respect to the b-operators. That is,
$$
H_{b}^{\ell}(\Sigma_{D})=\{u\in L^{2}_{b}(\Sigma_{D})|V u \in L^{2}_{b}(\Sigma_{D}),\  \forall V\in \Diff^{\ell}_{b}(\Sigma_{D})
\}.
$$
Using such b-based spaces for conic operators has certain advantages, as these functions satisfy dilation invariance properties.

There is a well-developed theory of self-adjoint extensions of symmetric operators in the setting of manifolds with conical singularities, c.f. ~\cite{Lesch, GM, GKM2, GKM}. For the case of Laplace--Beltrami operators in this paper, we also refer to~\cite{Hill, HK} for the theory on flat conical surfaces.  Since such extensions only concern  the local behavior near each cone point and the leading part of $\Delta_{g}$ is the same as in the flat case,
the expansions later in this section follow from exactly the same computation.

The closure of $\Delta_{g}$ in $\frakr^{-1}L^{2}_{b}(\Sigma_{D})$ with respect to the graph norm is a symmetric operator
\begin{multline*}
\Delta_{g}^{\min}: \calD^{\min} \rightarrow \frakr^{-1}L^{2}_{b}, \\ \calD^{\min}=\overline{\calC_{c}^{\infty}(\Sigma\setminus {\rm supp}\, D)} \text{ with respect to } \|u\|_{\frakr^{-1}L^{2}_{b}}+\|\Delta_{g}u\|_{\frakr^{-1}L^{2}_{b}},
\end{multline*}
while there is another domain
\begin{multline*}
\Delta_{g}^{\max}: \calD^{\max} \rightarrow \frakr^{-1}L^{2}_{b}, \\ \calD^{\max}=\{u\in \frakr^{-1}L^{2}_{b}: \Delta_{g}u\in \frakr^{-1}L^{2}_{b} \text{ in the distributional sense, i.e. }\\
\exists\phi \in \frakr^{-1}L^{2}_{b}, \text{ s.t. }\forall v\in \calD^{\min}, (\Delta_{g}v, u)=(v,\phi)\}.
\end{multline*}
In other words, $\calD^{\max}$ is the dual space of $\calD^{\min}$ with respect to the $L^{2}$ product.

There is a complete description of $\calD^{\min}$ and $\calD^{\max}$ in~\cite{GM}. In particular, we have
\begin{proposition}[{\cite[Lemma 3.5, Proposition 3.6, Lemma 3.11]{GM}}]
The minimal and maximal domains of $\Delta_{g}$ satisfy the following conditions:
\begin{enumerate}
\item $\calD^{\min}=\calD^{\max} \bigcap \big(\cap_{\epsilon>0} \frakr^{1-\epsilon}H_{b}^{2}(\Sigma_{D})\big)$;
\item $\frakr H_{b}^{2}(\Sigma_{D})\subset \calD^{\min}$;
\item $\exists \epsilon>0$ such that $\calD^{\max} \hookrightarrow \frakr^{-1+\epsilon} H_{b}^{2}(\Sigma_{D})$.
\end{enumerate}
\end{proposition}

Near a cone point of angle $2\pi\beta$, we can write out the expansion of an element in $\calD^{\max}$ as following:
\begin{multline}
u_{\max}= a_{0}+ b_{0} \log \frakr \\+ \sum_{1\leq |k|\leq J} |k|^{-1/2}a_{k} \frakr^{|k|/\beta}e^{ik\theta}+\sum_{1\leq |k|\leq J}|k|^{-1/2} b_{k}\frakr^{-|k|/\beta} e^{ik\theta} + \tilde u,\\ \tilde u\in \calD^{\min}, \ a_{k}, b_{k} \in \CC
\end{multline}
where $J=[\beta]$ if $\beta\notin \NN$ and $J=\beta-1$ if $\beta\in \NN$. When there are multiple cone points, we use the notation $(a_{k}^{i}, b_{k}^{i})_{-J_{i}\leq k\leq J_{i}}$ for the expansion near $P_{i}$.
We refer to~\cite[Proposition 3.3]{Hill} for the explicit computation that justifies the above expansion. Note that if $\beta<1$, the only coefficients remaining are $(a_{0}, b_{0})$.

The classical Von Neumann theory~\cite{RS1, RS2} shows that any self-adjoint extension of $\Delta_{g}$ is a space between $\calD^{\min}$ and $\calD^{\max}$,  and has a one-to-one correspondence with the Lagrangian in the space of coefficients
$$\bigcup_{\substack{0\leq |k|\leq J_{i}\\1\leq i\leq n}}\{a_{k}^{i}, b_{k}^{i}\}\in \CC^{2J}, \, J=\sum_{i=1}^{n} (2J_{i}+1).$$
Here the symplectic pairing is given by
\begin{equation}
\begin{split}
&\Omega(\vec A, \vec A'):=\langle A_{+}, A'_{-}\rangle- \langle A_{-}, A'_{+}\rangle = \sum  (a_{k}^{i}\overline{ {b_{k}^{i}}' }- b_{k}^{i} \overline{{a_{k}^{i}}'}), \\
& \vec A=\left(a_{-J_{1}}^{1}, \dots, a_{k}^{i}, \dots,  a_{J_{n}}^{n},b_{-J_{1}}^{1}, \dots, b_{k}^{i}, \dots,  b_{J_{n}}^{n}\right)=:(A_{+}, A_{-}).
\end{split}
\end{equation}
where $\langle \cdot, \cdot\rangle$ is the inner product in $\CC^{J}$.

In particular, there are two self-adjoint extensions we are going to use:
\begin{definition}\label{d:Dfr}
The domain of the Friedrichs extension $\calD^{\Fr}$ consists of all bounded elements $u\in \calD^{\max}$, i.e. any function with an expansion
$$
u=a_{0}+\sum_{1\leq |k|\leq J} |k|^{-1/2}a_{k} \frakr^{|k|/\beta}e^{ik\theta} +\tilde u, \ \tilde u\in \calD^{\min}, \ a_{k} \in \CC.
$$
\end{definition}
\begin{definition}\label{d:Dhol}
The domain of holomorphic extension $\calD^{\Hol}$  consists of functions with expansion
\begin{multline}
u= a_{0}+\sum_{1\leq k\leq J} |k|^{-1/2}a_{k} \frakr^{|k|/\beta}e^{ik\theta} + \sum_{1\leq -k\leq J} |k|^{-1/2}b_{k} \frakr^{-|k|/\beta}e^{ik\theta}+ \tilde u,\\
 \tilde u\in \calD^{\min}, \ a_{k}, b_{k}\in \CC.
\end{multline}
\end{definition}
We denote by $\Delta_{g}^{\Fr}$ and $\Delta_{g}^{\Hol}$ the two self-adjoint operators associated to the domain $\calD^{\Fr}$ and $\calD^{\Hol}$.

In terms of complex coordinate $z=|z|e^{i\theta}$ where $|z|\sim \frakr^{1/\beta}$, the two expansions above can be rewritten as
\begin{align}
u\in \calD^{\Fr}\Leftrightarrow u=a_{0}+ \sum_{1\leq k\leq J} (a_{k}z^{k} +b_{k}\bar z^{k})+\calO(|z|^{J+1}), \label{e:extensionzF}\\
u\in \calD^{\Hol}\Leftrightarrow u=a_{0}+\sum_{1\leq |k|\leq J} a_{k} z^{k} +\calO(|z|^{J+1}).\label{e:extensionzH}
\end{align}
Notice that if all $\beta_{i}<1$, then  the only nontrivial coefficient is $a_{0}$, in which case $\calD^{\Fr}=\calD^{\Hol}$.

\section{A 2-eigenfunction in the domain}\label{s:eigen}
In this section we construct a 2-eigenfunction for reducible metrics.
\begin{theorem}\label{t:eigen}
If $g$ is a spherical conical metric on $\Sigma$ with reducible monodromy, then $2\in \Spec(\Delta_{g}^{\Fr})\cap \Spec(\Delta_{g}^{\rm Hol})$.
\end{theorem}

Let $g$ be a reducible metric representing $D$ on $\Sigma$. Then there exists
a developing map of $g$, written as $f:\Sigma\backslash {\rm supp}\, D\to {\overline {\Bbb C}}$, such that the monodromy of $f$ is contained in $U(1)$.
Hence the metric $g$ can be written as the pullback by $f$ of the standard spherical metric $g_{\rm st}=\frac{4|dw|^2}{(1+|w|^2)^2}$, i.e.
\begin{equation}\label{e:gst}
g=\frac{4|f'|^{2}}{(1+|f|^{2})^{2}}|dz|^{2}.
\end{equation}
Its Laplace--Beltrami operator is then given by
\begin{equation}\label{e:Deltag}
\Delta_{g}=-\frac{(1+|f|^{2})^{2}}{|f'|^{2}}\partial_{z}\partial_{\bar z}.
\end{equation}

Consider the following function on $\Sigma\setminus {\rm supp}\, D$:
\begin{equation}\label{e:eigen}
\phi=\frac{1-|f|^{2}}{1+|f|^{2}}.
\end{equation}
Since $f$ is a multi-valued developing map, $\phi$ is also potentially multi-valued on $\Sigma\setminus {\rm supp}\, D$. Moreover, a priori $\phi$ is only defined on $\Sigma\setminus {\rm supp}\, D$. However, we have the following lemma for reducible metrics:
\begin{lemma}
If $g$ has reducible monodromy, then $\phi$ is a well-defined single-valued function on $\Sigma$.
\end{lemma}
\begin{proof}
Take a representative $\mathfrak{f}$ of $f$ outside the branching points of $f$. Since the monodromy of $f$ is contained in $U(1)$, it is  straight-forward to check that $|\mathfrak{f}|$, and hence $\phi$, is single-valued on $\Sigma$ outside the branching points of $f$.  In addition, the definition of $\phi$ is independent of choice of $\mathfrak{f}$. By \cite[Theorem 1.4]{CWWX2015}, $\phi$ extends continuously to $\Sigma$.
\end{proof}

\begin{remark}
When $g$ is irreducible, the function $\phi$ defined above is {\it not} single-valued.
\end{remark}

\begin{lemma}
Near a cone point of angle $2\pi\beta$, if $\beta\notin\NN$, then there exists one complex coordinate $z$ and one representative $\frakf$ of $f$ such that $\frakf=z^{\beta}$; if $\beta\in \NN$, then there exists one complex coordinate $z$, a $\rm{PSL}(2,\CC)$ matrix $\left(
\begin{array}{cc}
a & b\\
c & d
\end{array}
\right)$ and one representative $\frakf$ of $f$ such that $\frakf=\frac{az^{\beta}+b}{cz^{\beta}+d}$.
\end{lemma}
\begin{proof}
It follows from the proof of Theorem 1.4 in~\cite{CWWX2015}.
\end{proof}

\begin{lemma}
The function $\phi$ defined in~\eqref{e:eigen} is contained in $\calD^{\Fr}\cap \calD^{\Hol}$.
\end{lemma}
\begin{proof}
Near a cone point with non-integer cone angle, using the expression $f= z^{\beta}$, we get the local expression of $\phi$ as
$$
\phi=\frac{1-|z|^{2\beta}}{1+|z|^{2\beta}}
$$
and hence
$\phi$  has the following expansion
$$
\phi\sim 1-2|z|^{2\beta}+\calO(|z|^{2\beta+\epsilon}).
$$

Near a cone point with integer cone angle $2\pi n$, the developing map $f$ has a different expression $f=\frac{az^{n}+b}{cz^{n}+d}$ where $ad-bc=1$, and hence $\phi$ is given by
\begin{multline}
\phi\sim \frac{|b|^{2}-|d|^{2}}{|b|^{2}+|d|^{2}} -
\frac{ 2\bar b\bar d}{(|b|^{2}+|d|^{2})^{2}} z^{n} -\frac{ 2 b d}{(|b|^{2}+|d|^{2})^{2}}  \bar z^{n}\\
 + \frac{2\bar b \bar d (a\bar b + c\bar d)}{(|b|^{2}+|d|^{2})^{3}} z^{2n} + \frac{2 b d (\bar a b + \bar c d)}{(|b|^{2}+|d|^{2})^{3}} \bar z^{2n}+ \frac{2(|b|^{2}-|d|^{2})(|a|^{2}+|c|^{2})}{(|b|^{2}+|d|^{2})^{2}}|z|^{2n}\\
+ \calO(|z|^{2n+\epsilon}).
\end{multline}

In either case, by comparing with~\eqref{e:extensionzF} and~\eqref{e:extensionzH}, we see that $\phi$ is in $\calD^{\Fr}\cap \calD^{\Hol}$.
\end{proof}
Using the explicit expression of $\Delta_{g}$ in~\eqref{e:Deltag}, it is easy to check that $\phi$ satisfies:
$$
\Delta_{g}\phi=2\phi.
$$
This completes the proof of Theorem~\ref{t:eigen}.

\begin{ex}
\label{ex:three}
Consider the standard sphere $\bS^{2}$ with the spherical metric $g_{\rm st}=d\frakr^{2}+\sin^{2}\frakr d\theta^{2}$, which is related to the metric~\eqref{e:gst} in conformal coordinates (where the developing map is $f(z)=z, \ z=re^{iw}$) by the following relation
$$
\frakr=2\arctan r,  \ \theta= w.
$$
Then one of the 2-eigenfunctions of $g_{\rm st}$ is given by $\cos r$, which is exactly the same as $\phi$ in~\eqref{e:eigen}. The other two eigenfunctions
$\sin r \cos \theta$ and $\sin r \sin \theta$
are identified with the real and imaginary parts of
$$
\frac{2f}{1+|f|^{2}}.
$$
\end{ex}

\section{A spectral condition for reducible metrics}\label{s:sp}
In this section we prove the following:
\begin{theorem}\label{t:reducible}
If there is a real-valued function $\phi\in \calD^{\Hol}$ satisfying $\Delta_{g}\phi=2\phi$, then $g$ is a reducible metric.
\end{theorem}
Let $\phi \in \calD^{\Hol}$ be an eigenfunction with $\Delta_{g}\phi=\lambda\phi$. Define the complex gradient of $\phi$ as
$$
X=\phi^{,z} \partial_{z}.
$$
We will use this vector field to show $\lambda\geq 2$, and if $\lambda=2$ then $X$ is a meromorphic vector field, which will then be related to a developing map of $g$ provided
that the 2-eigenfunction $\phi$ is real-valued. Before the proof, we first give two examples.
\begin{ex}
Take the spherical football with angles $(2\pi\beta, 2\pi\beta)$ where $\beta\notin\NN$, and let $z$ be the (global) conformal coordinate centered at one of the cone points. The eigenfunction defined in~\eqref{e:eigen} is given by $\frac{1-|z|^{2\beta}}{1+|z|^{2\beta}}$, and its corresponding gradient vector field is given by
$$X=-z\partial_{z}.$$
\end{ex}

\begin{ex}
Take the three 2-eigenfunctions on $\bS^{2}$ described in Example~\ref{ex:three}, then the three corresponding vector fields are holomorphic on the sphere:
\begin{equation}
X_{1}=-\frac{1}{2}z\partial_{z}, \ X_{2}=\frac{1}{4} (1-z^{2})\partial_{z}, \ X_{3}=\frac{1}{4}i (1+z^{2})\partial_{z}.
\end{equation}
Consider the double cover of a sphere using the pullback by $f:z\rightarrow z^{2}$, which gives a spherical conical metric with two antipodal cone points each with angle $4\pi$. The three eigenfunctions from the sphere lift to the double cover, and the three meromorphic vector fields are given by
\begin{equation}
X_{1}=-\frac{1}{4}z\partial_{z}, \ X_{2}=(\frac{1}{8}z^{-1}-\frac{1}{8}z^{3})\partial_{z}, \ X_{3}=i(\frac{1}{8}z^{-1}+\frac{1}{8}z^{3})\partial_{z}.
\end{equation}
\end{ex}

\begin{remark}
Each of the meromorphic vector fields above can be viewed as a generator of an infinitesimal diffeomorphism on the twice-punctured sphere. In particular, $X_{1}$ is
associated to the conformal dilations on the sphere, and generates a diffeomorphic family of spherical conical metrics with the same conical data.
Such conformal dilations can be seen as gauge actions, which shows that the presence of 2-eigenfunctions creates obstructions in solving the nonlinear Liouville equation~\eqref{e:liou} that was studied in detail in~\cite{MaZh19}.
\end{remark}

\subsection{A meromorphic vector field}\label{ss:merovec}
We will prove the following theorem:
\begin{theorem}\label{t:lambda}
For any spherical conical metric $g$, the first nonzero eigenvalue of $\Delta_{g}^{\Hol}$ satisfies $\lambda_{1}\geq 2$.
\end{theorem}
If the equality holds, then we have the following proposition for any 2-eigenfunction in the holomorphic extension.
\begin{proposition}\label{p:hol}	
Let $\phi\in \calD^{\Hol}$ be an eigenfunction (not necessarily real-valued) satisfying $\Delta_{g}\phi=2\phi$, then its complex gradient vector field defined as $X=\phi^{,z}\partial_{z}$ is meromorphic on $\Sigma$. Moreover, $X$ has the following properties:
\begin{enumerate}
\item At any point $p\notin {\rm supp}\, D$, $X$ is holomorphic;
\item At a point $p \in {\rm supp}\, D$ with $\beta\notin\NN$ and $\beta<1$, $X$ has a zero;
\item At a point $p \in {\rm supp}\, D$ with $\beta>1$,
$X$ can have a pole of order at most $[\beta]-1$.
\end{enumerate}
\end{proposition}

We start with an arbitrary eigenfunction $\phi \in \calD^{\Hol}$ satisfying
$$
\Delta_{g}\phi=\lambda \phi, \ \lambda> 0.
$$
Locally the metric is given by $g=e^{2u}|dz|^{2}$ as in~\eqref{e:gst}, then we can write the gradient vector field $X$ of $\phi$ as
$$
X=e^{-2u}\frac{\partial \phi}{\partial\bar z}\partial_{z}=\frac{4|f'|^2}{(1+|f|^2)^2}\,\frac{\partial \phi}{\partial\bar z}\partial_{z}.
$$
Away from cone points,  by elliptic regularity it is easy to see that $\phi$ is smooth. Hence $X$ is a smooth vector field on $\Sigma \setminus {\rm supp}\,D$.

Using Bochner's identity (see for example~\cite[1.38]{Ballmann}), we have the following pointwise identity for $X$:
\begin{equation}
\nabla^{*}\nabla X=\lambda X- \Ric X
\end{equation}
where $\nabla^{*}$ is the formal adjoint operator of $\nabla$. Recall for the spherical metric $g$ considered here, $\Ric X=X$. 
Hence by~\cite[(4.80)]{Ballmann},
\begin{equation}\label{e:Bochner}
\nabla^{*}\nabla^{(0,1)}X=\frac{1}{2}(\nabla^{*}\nabla X-\Ric X)=\frac{1}{2}(\lambda-2)X.
\end{equation}

We now show that the following integration by parts is valid:
\begin{equation}\label{e:ibp}
\int_{\Sigma} |\nabla^{(0,1)}X|^{2}=\int_{\Sigma} (\nabla^{*}\nabla^{(0,1)}X, X).
\end{equation}
Note that all we need to show is that the integral of the left hand side converges, in other words $\nabla^{(0,1)}X$ is in $L^{2}$. Then using $\nabla X=\nabla^{(1,0)}X + \nabla^{(0,1)}X$ and pointwise $(\nabla^{(1,0)}X, \nabla^{(0,1)}X)=0$ one can apply integration by parts to get the expression on the right hand side. If \eqref{e:ibp} holds, then from~\eqref{e:Bochner} we immediately get $\lambda\geq 2$, and $\lambda=2$ if and only if $\nabla^{(0,1)}X=0$.

 To check that $\nabla^{(0,1)}X$ is indeed square integrable, we compute the decay rate of $\nabla^{(0,1)}X$ near each cone point.
To do this we take $z=re^{i\theta}$ to be the complex coordinate near a cone point of angle $2\pi \beta$, and decompose the eigenfunction $\phi$ locally into Fourier series  $\phi=\sum_{k\in \ZZ}\phi_{k}(r) e^{ik\theta}$. We compute the functions $\phi_{k}(r)$ and then use them to express $\nabla^{(0,1)}X$. If all $\phi_{k}$ decay fast enough near the cone points, then $\nabla^{(0,1)}X$ will be square integrable. There is a slight difference between the two cases $\beta\notin \NN$ and $\beta\in \NN$, so we carry out the computation in the first case in details and then point out the difference in the second case.

\subsubsection{Developing maps near a cone point of $\beta \notin \NN$}
Near such a point the developing map $f$ can be written as
$$
f=z^{\beta},
$$
so the metric is given by
$$
g=\frac{4\beta^{2}|z|^{2(\beta-1)}}{(1+|z|^{2\beta})^{2}}|dz|^{2}.
$$
Recall $z=re^{i\theta}$, then the equation $\Delta_{g}\phi=\lambda\phi$ is given by
$$
-\frac{(1+r^{2\beta})^{2}}{4 \beta^{2}r^{2(\beta-1)}}\frac{1}{r^{2}}\big((r\partial_{r})^{2}+\partial_{\theta}^{2}\big)\phi = \lambda\phi.
$$
Write $\phi=\sum_{k=-\infty}^{\infty} \phi_{k}(r)e^{ik\theta}$, the Fourier decomposition of the equation gives a sequence of regular-singular ODEs:
$$
-\frac{(1+r^{2\beta})^{2}}{4\beta^{2}r^{2(\beta-1)}}\frac{1}{r^{2}}\big((r\partial_{r})^{2}-k^{2}\big)\phi_{k}=\lambda\phi_{k}, \ k\in \ZZ.
$$
Each ODE has two linearly independent solutions, with leading term $r^{k}$ and $r^{-k}$ respectively. By choosing $\phi\in \calD^{\Hol}$, we require that for $-J\leq k\leq J$,
each $\phi_{k}$ has an expansion with leading term $r^{k}$. Putting the leading term $r^{k}$ into the equation we see that in order to match the right hand side the next term should be given by $r^{k+2\beta}$. This gives an $r^{k+2\beta}$ term to match on the right hand side, hence the next term in the expansion is given by $r^{k+4\beta}$. Iteratively
we get
\begin{equation}\label{e:phik}
\phi_{k}(r)=C_{k,0}r^{k}+C_{k,1}r^{k+2\beta}+C_{k,2}r^{k+4\beta}+\dots+C_{k,j}r^{k+2j\beta}+\dots,
\end{equation}
where $C_{k,i}$ are determined iteratively. In particular, if $\lambda=2$, then there is the following iteration:
\begin{multline}
-8\beta^{2}C_{k,j-1}=\big((k+2j\beta)^{2}-k^{2}\big)C_{k,j}\\+2\big((k+2(j-1)\beta)^{2}-k^{2}\big)C_{k,j-1}+ \big((k+2(j-2)\beta)^{2}-k^{2}\big)C_{k,j-2}.
\end{multline}
(When $j=1$ the $C_{k,j-2}$ term is removed.)
It is straightforward to check that if $\lambda=2$ then $C_{k,j}$ are actually given by
\begin{equation}\label{e:Ckj}
C_{k,j}=(-1)^{j}\frac{2\beta}{k+\beta}C_{k,0},\ j=1, 2, \dots.
\end{equation}

On the other hand, for $|k|>J$, the assumption that $\phi$ is in $L^{2}(\Sigma, d\Vol_{g})$ requires the solution $\phi_{k}$ to start with $r^{|k|}$:
\begin{equation}
\phi_{k}(r)=C_{k,0}r^{|k|}+C_{k,1}r^{|k|+2\beta}+C_{k,2}r^{|k|+4\beta}+\dots+C_{k,j}r^{|k|+2j\beta}+\dots.
\end{equation}
where the coefficients $C_{k,j}, j>0$ can be determined again by $C_{k,0}$ by the same fashion. In particular when $\lambda=2$, we have a similar expression as~\eqref{e:Ckj}, except $k$ is replaced by $|k|$.

Now we compute $X$ and $\nabla^{(0,1)}X$ using this local expansion.

\noindent\textbf{The vector field $X$}

Compute the complex gradient of $\phi$ as
$$X=\bigg(e^{-2u}\frac{\partial \phi}{\partial\bar z}\bigg)\partial_{z}=\Bigg(\sum_{k\in \ZZ}X_{k}(r)e^{i(k+1)\theta}\Bigg)\partial_{z}.$$
Here we use the local expression of $g$ and
$
\partial_{\bar z}=\frac{e^{i\theta}}{2}\big(\partial_{r}-\frac{1}{ir}\partial_{\theta}\big)
$
to get
\begin{equation}
X_{k}=e^{-i(k+1)\theta}\frac{(1+r^{2\beta})^{2}}{4\beta^{2}r^{2(\beta-1)}}\frac{e^{i\theta}}{2}(\partial_{r}-\frac{1}{ir}\partial_{\theta}) (\phi_{k}(r)e^{ik\theta}).
\end{equation}
It simplifies to the following:
$$
X_{k}= \frac{(1+r^{2\beta})^{2}}{8\beta^{2}r^{2(\beta-1)}} \bigg(\phi_{k}' -\frac{k}{r}\phi_{k}\bigg).
$$
Recall when $-J\leq k\leq J$, $\phi_{k}$ is given by~\eqref{e:phik}. The first term in $\phi_{k}$ is $C_{k,0}r^{k}$, which is eliminated by the operator $\partial_{r} - k/r$. Therefore
$$
\phi_{k}' -\frac{k}{r}\phi_{k}=\sum_{j\geq 1} 2j\beta C_{k,j}r^{k+2j\beta-1}.
$$
This gives
\begin{equation}\label{e:Xk}
X_{k}= \frac{(1+r^{2\beta})^{2}}{8\beta^{2}}\sum_{j\geq 1} 2j\beta C_{k,j}r^{k+(2j-2)\beta+1}=\frac{1}{8\beta^{2}}\sum_{j\geq 1}\tilde C_{k,j} r^{k+(2j-2)\beta+1},
\end{equation}
where $\tilde C_{k,j}=2j\beta C_{k,j}+4(j-1)\beta C_{k,j-1}+ 2(j-2)\beta C_{k,j-2}$ (for $j=1$ the last term is removed). The same computation above applies to $k>J$ as well.

On the other hand, when $k<-J$ the first term in $\phi_{k}$ is not eliminated and we get
$$
\phi_{k}' -\frac{k}{r}\phi_{k}=\sum_{j\geq 0} (-2k+2j\beta) C_{k,j}r^{|k|+2j\beta-1},
$$
which gives
\begin{equation}\label{e:Xkbig}
X_{k}= \frac{(1+r^{2\beta})^{2}}{8\beta^{2}}\sum_{j\geq 0} (-2k+2j\beta) C_{k,j}r^{|k|+(2j-2)\beta+1}=\frac{1}{8\beta^{2}}\sum_{j\geq 0} \tilde C_{k,j}r^{|k|+(2j-2)\beta+1},
\end{equation}
where $\tilde C_{k,j}$ satisfies $$\tilde C_{k,j}=(-2k+2j\beta) C_{k,j}+2\big(-2k+2(j-1)\beta\big) C_{k,j-1}+ \big(-2k+2(j-2)\beta\big) C_{k,j-2}.$$
Again for $j=1$ the last term is removed.

Here we give the following observation of $\tilde C_{k,j}$, which will be used later.
\begin{lemma}\label{l:Ckj}
If $\lambda=2$, then the coefficients $\tilde C_{k,j}$ in~\eqref{e:Xk} and~\eqref{e:Xkbig} satisfy
\begin{equation}
\begin{split}
\tilde C_{k,j}=0, \ \forall j\geq 2.
\end{split}
\end{equation}
\end{lemma}
\begin{proof}
It follows directly by substituting~\eqref{e:Ckj} into the expression of $\tilde C_{k,j}$.
\end{proof}

\noindent\textbf{Expression of $\nabla^{(0,1)}X$}

Now we compute $\nabla^{(0,1)}X$ which is a (1,1)-tensor, locally given by
$$\sum_{k\in \ZZ} \partial_{\bar z}\big(X_{k}e^{i(k+1)\theta}\big) \partial_{z}\otimes d\bar z.$$
Using the expression of $X_{k}$ obtained above,
we compute $\partial_{\bar z}(X_{k}e^{i(k+1)\theta})=\frac{e^{i\theta}}{2} (\partial_{r}-\frac{1}{ir}\partial_{\theta}) (X_{k}(r)e^{i(k+1)\theta})$ for each $k$.

When $-J\leq k\leq J$,  the first term in~\eqref{e:Xk} is given by $\frac{1}{8\beta^{2}}\tilde C_{k,1}r^{k+1}$ and it is again eliminated by the operator $\partial_{r} - (k+1)/r$.
So we have
\begin{equation}\label{e:dXkgeneral}
\partial_{\bar z}\big(X_{k}e^{i(k+1)\theta}\big)=\frac{e^{i(k+2)\theta}}{16\beta^{2}}\sum_{j\geq 2} (2j-2)\beta\tilde  C_{k,j} r^{k+(2j-2)\beta},
\end{equation}
Since $J= [\beta]$ and $\beta$ is not integer, we have $-\beta<J$. Therefore we have a leading term $r^{2\beta+k}$ where $2\beta+k\geq \beta>0$. Moreover, if $\lambda=2$,
then by Lemma~\ref{l:Ckj}, we simply have
\begin{equation}\label{e:dXk}
\partial_{\bar z}\big(X_{k}e^{i(k+1)\theta}\big) =\frac{e^{i(k+2)\theta}}{16\beta^{2}}\sum_{j\ge 2}\tilde  C_{k,j} (2j-2)\beta r^{k+(2j-2)\beta}=0.
\end{equation}
Note that when $k>J$, the same computation as in~\eqref{e:dXkgeneral} and~\eqref{e:dXk} applies.

On the other hand, when $k<-J$, the leading term in~\eqref{e:Xkbig} is not eliminated and we get
\begin{equation}\label{e:partial}
\begin{split}
\partial_{\bar z}\big(X_{k}e^{i(k+1)\theta}\big) &=\frac{e^{i(k+2)\theta}}{16\beta^{2}}\sum_{j\geq 0}\tilde C_{k,j}\big(2|k|+(2j-2)\beta\big) r^{|k|+(2j-2)\beta}.
\end{split}
\end{equation}
Since $\beta\notin \NN$, we have $|k|>J>\beta$, so the leading term in~\eqref{e:partial} is given by $r^{|k|-2\beta}$ which satisfies $|k|-2\beta>-\beta$. Recall that $r=|z|$, so in terms of the geodesic coordinates $(\frakr, \theta)$ for which the associated volume form is $\frakr d\frakr d\theta$, the computation above implies that each term decays slower than $\frakr^{-1}$.


Combining the analysis above, we have justified the integration by parts in~\eqref{e:ibp} near a cone point with $\beta\notin \NN$.

\subsubsection{Developing maps near a cone point with $\beta=n\in \NN$}

Near such a cone point the developing map $f$ is given by
$$
f=\frac{az^{n}+b}{cz^{n}+d}
$$
for some $a,b,c,d\in \CC$ with $ad-bc=1$. By choosing a suitable coordinate and a representative of the developing map the metric is given by the same form as the non-integer case~\cite{FSX}:
$$
g=\frac{4n^{2}|z|^{2(n-1)}}{(1+|z|^{2n})^{2}}|dz|^{2}.
$$
Therefore we can again obtain the expansion of $\phi_{k}$ as following:
\begin{equation}\label{e:phikint}
\begin{split}
\phi_{k}=\sum_{j\geq 0}C_{k,j} r^{k+2jn}, -J\leq k; \\
\phi_{k}=\sum_{j\geq 0}C_{k,j} r^{|k|+2jn}, -J>k.
\end{split}
\end{equation}

The computation of $X$ and $\nabla^{(0,1)}X$ can now be applied verbatim, and we list the results here. First of all
\[
X=\Bigg(\sum_{k\in \ZZ} X_{k}(r) e^{i(k+1)\theta}\Bigg)\partial_{z}.
\]
In terms of local coordinates it is given by
\begin{equation}
X_{k}=e^{-i(k+1)\theta}\frac{(1+r^{2n})^{2}}{4n^{2}r^{2(n-1)}}\frac{e^{i\theta}}{2}
\bigg(\partial_{r}-\frac{1}{ir}\partial_{\theta}\bigg) \big(\phi_{k}(r)e^{ik\theta}\big).
\end{equation}
It immediately follows that
\begin{equation}\label{e:Xkint}
\begin{split}
X_{k}(r)=\sum_{j\geq 1}\tilde C_{k,j} r^{k+(2j-2)n+1}, \ -J\leq k,\\
X_{k}(r)=\sum_{j\geq 0}\tilde C_{k,j}r^{|k|+(2j-2)n+1}, \ -J>k.
\end{split}
\end{equation}

We then compute $\nabla^{(0,1)}X$ which is locally given by
$$
\sum_{k\in \ZZ} \partial_{\bar z}\big(X_{k}e^{i(k+1)\theta}\big) \partial_{z}\otimes d\bar z.
$$
As before we compute $\partial_{\bar z}\big(X_{k}e^{i(k+1)\theta}\big)=\frac{e^{i\theta}}{2} (\partial_{r}-\frac{1}{ir}\partial_{\theta}) (X_{k}(r)e^{i(k+1)\theta})$ to get
\begin{equation}\label{e:partialXkint}
\begin{split}
\partial_{\bar z}(X_{k}e^{i(k+1)\theta})=e^{i(k+2)\theta}\sum_{j\geq 2} (2j-2)n \tilde  C_{k,j} r^{k+(2j-2)n}, \ -J\le k,\\
\partial_{\bar z}(X_{k}e^{i(k+1)\theta})=e^{i(k+2)\theta}\sum_{j\geq 0} \big(-2k+(2j-2)n\big)\tilde C_{k,j} r^{|k|+(2j-2)n}, \ -J>k.
\end{split}
\end{equation}

Notice that $J=n-1$, so we can check again that in both cases above the term has enough decay to be integrable. In the first case the smallest exponent $k+2n>0$. In the second case, we have $|k|-2n>-n$ except when $k=-(J+1)=-n$. If $|k|-2n>-n$, then again notice that $r^{-n}\sim \frakr^{-1}$ so the term is integrable. If $k=-n$, then we have $-2k-2n=0$. In this case the leading term vanishes in $\nabla^{(0,1)}X$ and the next term in the expansion is integrable again.

\subsubsection{Justification of the integration by parts}
Now we are ready to prove Theorem~\ref{t:lambda} and Proposition~\ref{p:hol}.
\begin{proof}[Proof of Theorem~\ref{t:lambda}]
From the previous computation of $\nabla^{(0,1)}X$ near each conical point, we know that $|\nabla^{(0,1)}X|^{2}$ is integrable. Therefore using~\eqref{e:Bochner} and~\eqref{e:ibp} we have
\begin{equation}
0\leq\int |\nabla^{(0,1)}X|^{2} = \frac{1}{2}(\lambda-2)\int |X|^{2},
\end{equation}
which immediately shows that $\lambda\geq 2$.
\end{proof}
Now we look at the case when $\lambda=2$.
\begin{proof}[Proof of Proposition~\ref{p:hol}]
We combine~\eqref{e:ibp} and~\eqref{e:Bochner} to get
\begin{equation}\label{e:vanish}
\nabla^{(0,1)}X=0
\end{equation}
which shows that $\nabla^{(0,1)}X$ is holomorphic on $\Sigma\setminus {\rm supp}\, D$.

We next show that $X$ can be extended to a meromorphic vector field on $\Sigma$. From the expansion of $X$ near each cone point (see expressions~\eqref{e:Xk} \eqref{e:Xkbig} and \eqref{e:Xkint}), we see that $X$ is bounded by $|z|^{-J}$, therefore there cannot be any essential singularities. Hence $X$ must be meromorphic on $\Sigma$.

Moreover, we can see the worst order of pole of $X$ from the expansion. For the behavior near a smooth point, $X$ is smooth since $\phi$ is smooth. Near a cone point, we look at the expansion of $X$ and $\nabla^{(0,1)}X$. If $\beta\notin \NN$, then all coefficients $\tilde C_{k,j}$ for $k<-J$ in~\eqref{e:partial} has to vanish because $\nabla^{(0,1)}X=0$, and this shows that
the worst decay of $X$ is given by the $k=-J$ mode, which is $r^{-J+1}$ by~\eqref{e:Xk}. In particular, if $\beta<1$, then the worst decay is $r^{1}$, so it is actually a zero for $X$; if $\beta>1$, then the order of the pole is given by $r^{-[\beta]+1}$. On the other hand,  if $\beta=n\in \NN$, then the coefficient $C_{-J-1, 0}$ in~\eqref{e:partialXkint} might not vanish (since this term corresponds to $k=-n$ and does not appear in $\nabla^{(0,1)}X$), therefore the worst possible decay in $X$ is of the order $r^{-n+1}$.
\end{proof}
We remark here that the worst order of decay of $X$ described in the proposition above applies to any 2-eigenfunction. In the next step we assume in addition that $\phi$ is real-valued, and the decay estimate will be improved by relating to developing maps.

\subsection{From vector fields to reducible metrics}\label{ss:oneform}
From now on we assume that the eigenfunction $\phi$ is real-valued.
\begin{lemma}
The algebraic dual one-form of X, denoted by $\Omega$, is a meromorphic 1-form on $\Sigma$.
\end{lemma}
\begin{proof}
This follows from Proposition~\ref{p:hol} that $X$ is meromorphic.
\end{proof}

We denote by $D_X:=(\Omega)$ the divisor associated to the meromorphic one-form $\Omega$.
Let $(U,z)$ be a coordinate chart that does not intersect ${\rm supp}\, D_X\bigcup {\rm supp}\, D $.
Denote $g=e^{2u}|dz|^2$ on $U$, and let $X=\frac{1}{4}F(z)\partial_{z}$. That is, $F(z)=4e^{-2u}\phi_{\bar z}$.
\begin{lemma}
\label{lemma:exact}
There exists
a global positive constant $C=C_\phi$ which is independent of $U$, such that
\begin{equation}\label{e:C}
\phi_z=-\frac{\phi^2-C^2}{F(z)},\quad e^{2u}=-4\frac{\phi^2-C^2}{|F(z)|^2}.
\end{equation}
Moreover, the real part of the one-form $\Omega$ is exact on $\Sigma\backslash \big({\rm supp}\, D_X\bigcup {\rm supp}\, D \big)$.
\end{lemma}
\begin{proof}
Since $\phi$ is a 2-eigenfunction of $\Delta_g$, in local coordinates we have
$$-4e^{-2u}\phi_{z{\bar z}}=2\phi\quad\quad\quad \text{and}\quad\quad\quad
\phi_{z{\bar z}}=-\frac{\phi e^{2u}}{2}=-\frac{2\phi \phi_{\bar z}}{F}=-\bigg(\frac{\phi^2}{F}\bigg)_{\bar z}.$$
Since $F$ does not vanish anywhere in $U$, there exists a holomorphic function $g(z)$ on $U$ such that
\[\phi_z=-\frac{\phi^2}{F(z)}+g(z).\]
Since $\phi$ is real-valued, we also have
\[
\overline{\phi_{\bar z}}=-\frac{\phi^{2}}{{F(z)}}+{g(z)}.
\]
Combining with $F(z)=4e^{-2u}\phi_{\bar z}$, we find that
\[e^{2u}=\frac{4\overline{\phi_{\bar z}}}{\overline{F(z)}}=-\frac{4\phi^2}{|F(z)|^2}+\frac{4g(z)}{\overline {F(z)}}=-\frac{4\phi^2}{|F(z)|^2}+\frac{4g(z)F(z)}{|F(z)|^2}.\]
Therefore the
holomorphic function $4g(z)F(z)$ satisfies
$$
4g(z)F(z)=\big(|F(z)|^2 e^{2u}+4\phi^2\big),
$$
where the right hand side is a positive real function. It follows that $g(z)F(z)=C^{2}$ for some positive constant $C=C_U$ on $U$.
Since $\Sigma\backslash \big({\rm supp}\, D_X\bigcup {\rm supp}\, D\big) $ is connected, the constant $C=C_U$ does not depend on $U$. This proves~\eqref{e:C}.

By the first equality in \eqref{e:C}, we have
\begin{equation}\label{e:dzf}
2\Re\,\frac{dz}{F}=\frac{dz}{F}+\frac{d{\bar z}}{\bar F}=\frac{d\phi}{C^2-\phi^2}=d\left(\frac{1}{2C}\,\ln\,\frac{C+\phi}{C-\phi}\right).
\end{equation}
Since $\Omega=4\frac{dz}{F}$ and $C$ does not depend on $U$, we have that $\Re\,\Omega$ is exact on
$\Sigma\backslash \big({\rm supp}\, D_X\bigcup {\rm supp}\, D \big)$.
\end{proof}

\begin{lemma}
\label{lemma:char}
Define the following one-form on $\Sigma$
\begin{equation}\label{e:omega}
\omega:=-2C\frac{dz}{F}.
\end{equation}
Consider the multi-valued holomorphic function on $\Sigma\backslash \{\rm{poles\ of\ }\omega\}$ defined by
$$
f(z):=\exp\bigg(\int^z\,\omega\bigg),
$$
it satisfies the following:
$$
f:\Sigma\backslash \big({\rm supp}\, (\omega) \bigcup {\rm supp}\, D\big)  \rightarrow {\Bbb C}^*
$$
is a multi-valued holomorphic function with monodromy in $U(1)$.
Moreover, on $\Sigma\backslash \big({\rm supp}\, (\omega)\bigcup {\rm supp}\, D\big)$ we have
$$
f^*\,g_{\rm st}=g
$$
and
\begin{equation}
\label{equ:eigenf}
\phi=C\cdot\frac{1-|f|^2}{1+|f|^2}.
\end{equation}
\end{lemma}
\begin{proof}
We first observe that the divisor $D_X:=(\Omega)$ is equal to $(\omega)$, since $\Omega$ is a multiple of $\omega$.  The monodromy property of $f$ follows from the previous lemma that $\Re \omega$ is exact on
$\Sigma\backslash \big({\rm supp}\, (\omega)\bigcup {\rm supp}\, D\big)$.
Using~\eqref{e:dzf}, we have
\[|f|^2=\exp\bigg(\int^z\, 2\Re\,\omega\bigg)=\frac{C-\phi}{C+\phi},\]
which proves~\eqref{equ:eigenf}.

In order to show $f^*\,g_{\rm st}=g$, it suffices to show that in any complex coordinate chart $(U,z)$ not intersecting
${\rm supp}\, D_X\bigcup {\rm supp}\, D $,
we have
\[\frac{4|f'(z)|^2}{(1+|f|^2)^2}=4\frac{C^2-\phi^2}{|F|^2}.\]
Since $|f'(z)|^2=\frac{4C^2|f|^2}{|F|^2}$ and $|f|^2=\frac{C-\phi}{C+\phi}$,  the above equation follows from a direct computation.
\end{proof}


\begin{lemma}
\label{lemma:order}
The one-form $\omega$ defined in~\eqref{e:omega} has only simple poles with real residues, and its real part is exact outside its poles.
Moreover, the ${\Bbb R}$-divisor $D$ represented by the metric $g$ and the ${\Bbb Z}$-divisor $(\omega)$ associated to $\omega$ are related by
\[D=(\omega)_0+\sum_{P:\,{\rm pole\ of}\ \omega}\, \Big(|{\rm Res}_P(\omega)|-1\Big)[P],\]
where $(\omega)_0$ is the zero divisor of the meromorphic one-form $\omega$.
In particular, a cone point of non-integer angle $2\pi\beta$ must be a simple pole with residue $\pm \beta$ of $\omega$, and a cone point of integer angle $2\pi n$ may be either a simple pole with residue $\pm n$ or a zero of $\omega$ with multiplicity $n-1$.
\end{lemma}
\begin{proof} We divide the proof into the following three cases.

\noindent\textbf{Case 1: a smooth point}

Let $P$ be a smooth point of the metric $g$ where $X$ vanishes. We show that {\it $P$ is a simple pole of $\omega$ such that ${\rm Res}_P(\omega)$ equals either $-1$ or $1$.} Take an open disc $U$ centered at $P$ such that $U^*:=U\backslash \{P\}$ does not intersect ${\rm supp}\, (\omega)\bigcup {\rm supp}\, D$. Since $g$ is smooth on $U$, by \cite[Lemma 3.2]{CWWX2015}, we can choose a developing map of $g|_{U^{*}}$, denoted by $h:U^{*}\to \overline{\CC}$, and a suitable coordinate $z$ with $z(P)=0$, such that
$h(z)=\frac{az+b}{cz+d}$ with $ad-bc=1$. On the other hand, by Lemma \ref{lemma:char},
    the restriction $f|_{U^*}$ is a developing map of $g$. Since $f$ has trivial local monodromy around $P$, $f$ extends to $U$ and
    coincides with $h$ up to
  a  ${\rm PSU}(2)$ transformation. Therefore any representative ${\frak f}$ of the developing map $f$  is also given by the form ${\frak f}(z)=\frac{az+b}{cz+d}$ with $ad-bc=1$ on $U$. Hence on $U$
we have
    \[\omega=\frac{df}{f}=\frac{d{\frak f}}{\frak f}=\frac{dz}{(az+b)(cz+d)}.\]
    Since $\omega$ has a pole at $z=0$, we have that $bd=0$ but $b$ and $d$ cannot both be 0 because $ad-bc=1$, which implies that $P$ is a simple pole of $\omega$ and has residue $\pm 1$.

\noindent\textbf{Case 2: a cone point of non-integer angle}

Let $P$ be a cone point of angle $2\pi\beta$, where $\beta>0$ is {\it not} an integer. We show that {\it $P$ is a simple pole of $\omega$ such that ${\rm Res}_P(\omega)$ equals either $-\beta$ or $\beta$.}
   Similar to the previous case, we take an open disc $U$ centered at $P$ such that $U^*=U\backslash \{P\}$ does not intersect ${\rm supp}\, (\omega)\bigcup {\rm supp}\, D$. By Lemma \ref{lemma:char}, the restriction $f|_{U^*}$ is a developing map of $g|_{U^*}$.  Since $f|_{U^*}$ has monodromy in ${\rm U}(1)$ and $g|_{U^*}$ has a cone point of angle $2\pi\beta$ at $P$,  by the proof of Theorem 1.4 in \cite{CWWX2015}, there exists in $U$ a complex coordinate $z$ which is centered at $P$, such that in any disc contained in $U^*$, each representative ${\frak f}$ of $f$ has the form of either
    ${\frak f}(z)=\mu z^\beta$ or  ${\frak f}(z)=\tilde \mu z^{-\beta}$, where $\mu$ and $\tilde \mu$ are nonzero constants. Therefore, in a neighborhood of $P$ we have $\omega=\frac{d{\frak f}}{\frak f}=\pm \beta \frac{dz}{z}$.

\noindent\textbf{Case 3: a cone point of integer angle}

Let $P$ be a cone point of angle $2\pi n$, where $n>1$ is  an integer. We show that {\it $P$ is either a simple pole with residue $\pm n$ or a zero of $\omega$ with multiplicity $(n-1)$.}
Take an open disc $U$ centered at $P$ such that $U^*=U\backslash \{P\}$ does not intersect ${\rm supp}\, (\omega)\bigcup {\rm supp}\, D$. Since $g$ is smooth on $U^*$ and has a cone point with angle $2\pi n\in 2\pi{\Bbb Z}$ at $P$, by \cite[Lemma 3.2]{CWWX2015}, we can choose a developing map $h:U\to \overline{\CC}$ for $g|_U$ and a suitable coordinate $z$ with $z(P)=0$, such that $h$ is given by $h(z)=\frac{az^n+b}{cz^n+d}$ with $ad-bc=1$.
    On the other hand, by Lemma \ref{lemma:char},
    the restriction $f|_{U^*}$ is a developing map of $g|_{U^*}$. Since $f$ has trivial monodromy around $P$, $f$ extends to $U$ and
    coincides with $h$ up to a  ${\rm PSU}(2)$ transformation. Therefore each representative ${\frak f}$ of $f$ also has the form of ${\frak f}(z)=\frac{az^n+b}{cz^n+d}$ with $ad-bc=1$ on $U$. Hence
    on $U$ we have
    \[\omega=\frac{df}{f}=\frac{d{\frak f}}{\frak f}=\frac{n\,z^{n-1}}{(az^n+b)(cz^n+d)}\, dz.\]
    If $bd\not=0$, then $P$ is a zero of $\omega$ with multiplicity $(n-1)$. Otherwise, $P$ is a simple pole of $\omega$ with residue $\pm n$ .
\end{proof}

\begin{proof}[Proof of theorem~\ref{t:reducible}] It follows immediately from Lemma \ref{lemma:order}. Moreover, $\omega$ is
a character one-form \big(\cite[Definition 1.3]{CWWX2015}\big) of the reducible metric $g$.
\end{proof}

\begin{corollary}
The function $\phi$ extends continuously to $\Sigma$ and is smooth outside those poles with non-integer residues of $\omega$.
The positive constant $C$ in Lemma \ref{lemma:exact} equals $\max_\Sigma |\phi|$. Moreover,
$\phi$ achieves the maximum $C$ (resp. the minimum $-C$) at each pole of $\omega$ with positive (resp. negative) residue.
Each zero of $\omega$ is a saddle point of $\phi$.
\end{corollary}
\begin{proof}
It follows from \eqref{equ:eigenf}, Lemma \ref{lemma:order} and the local behaviour of $f$ near cone points described in the proof of Lemma~\ref{lemma:order}.
\end{proof}
\begin{remark}
This corollary shows that $\phi$ is also in the Friedrichs extension.
\end{remark}

\subsection{The dimension of the 2-eigenspace}\label{ss:dim}
\begin{theorem}
\label{thm:hol_eigenspace}
Let ${\Bbb E}_{\Bbb R}^{\rm Hol}$ be the real vector space of real 2-eigenfunctions of $\Delta_g^{\rm Hol}$ for a reducible metric $g$. Then $\dim\,{\Bbb E}_{\Bbb R}^{\rm Hol}$ equals either $1$ or $3$. Moreover,  $\dim\, {\Bbb E}_{\Bbb R}^{\rm Hol}=3$ if and only if $g$ is the pullback metric $f^*g_{\rm st}$ by a branched cover $f:\Sigma\to \overline{\CC}$.
\end{theorem}
\begin{proof} Choose a multiplicative developing map $f$ of the reducible metric $g$.

\noindent\textbf{Case 1: (nontrivial) reducible monodromy}

Suppose the monodromy of $f$ is non-trivial. Then
there are exactly two multiplicative developing maps, $f$ and $1/f$, for the metric $g$.
Since each real-valued 2-eigenfunction $\phi$ of $\Delta_g^{\rm Hol}$ can be expressed in terms of such a developing map as~\eqref{equ:eigenf},
we can see that the dimension of  ${\Bbb E}_{\Bbb R}^{\rm Hol}$ equals one.

\noindent\textbf{Case 2: trivial monodromy}

Suppose that $g$ is the pullback metric $f^*g_{\rm st}$ by a branched cover $f:\Sigma\to \overline{\CC}$.
First we know that $\dim\, {\Bbb E}_{\Bbb R}^{\rm Hol}\geq 3$, as the three eigenfunctions on $\bS^{2}$ lift to $\Sigma$ via the pullback, and they give three independent eigenfunctions.
On the other hand, for a real-valued 2-eigenfunction $\phi$ with maximum $1$, by \eqref{equ:eigenf}, there exists constants $a,\,b$ such that
\[|a|^2+|b|^2=1\quad\quad {\rm and}\quad\quad \phi=\frac{1-|g|^2}{1+|g|^2}\quad {\rm for}\quad g=\frac{af+b}{-\bar b f+\bar a}.\]
By a simple computation, we find that $\phi$ is a linear combination of the following three 2-eigenfunctions:
\[\frac{1-|f|^2}{1+|f|^2}, \quad \Re\,\frac{2f}{1+|f|^2}\quad{\rm and}\quad \Im\,\frac{2f}{1+|f|^2}.\]
This proves the dimension of ${\Bbb E}_{\Bbb R}^{\rm Hol}$ in this case is~3.
\end{proof}
\begin{remark}
In some cases one can obtain more information including all complex-valued 2-eigenfunctions. When $\Sigma$ is a branched cover of the sphere obtained by covering map $f(z)=z^{n}, n\in \NN$, explicit computation shows that the complex dimension of complex-valued 2-eigenfunctions in $\calD^{\Hol}$ and $\calD^{\Fr}$ are both equal to 3 (see~\cite[Lemma 2]{MaZh19} for the computation). In this case the eigenfunctions in the domain of the two extensions coincide and are given by the pullback by $f$ of the three eigenfunctions on the sphere.
\end{remark}

\section{Further discussion}\label{s:discussion}
The spectral condition in Theorem A uses a real eigenfunction in the holomorphic extension, which then is automatically a function in the Friedrichs extension, i.e. the coefficients $a_{k}$ in~\eqref{e:extensionzH} all vanish. However, one cannot replace the statement of the theorem by using a real eigenfunction in the Friedrichs extension. In fact, works by Mondello--Panov~\cite{MP2}, Eremenko--Gabrielov--Tarasov~\cite{EGT}, and Chen~\cite{ChenZ} suggest that there exist irreducible metrics with eigenvalue~2 in the spectrum of the Friedrichs Laplacian. It is ongoing work to find an explicit example of such an eigenfunction and understand the behavior of its associated complex gradient vector field, which is no longer guaranteed to be meromorphic.

One thing to notice here is that, even though the Laplace--Beltrami operator is real,  Friedrichs extension is the only extension that respects the real splitting. That is, we have the following relation for a function $\phi = u+iv$,
\[
\Delta_{g}\phi=2\phi \iff \Delta_{g}u=2u,\ \Delta_{g}v=2v,
\]
and
\[
\phi\in \calD^{\Fr} \iff u\in \calD^{\Fr},\ v\in \calD^{\Fr}.
\]
However for the holomorphic extension we only have
\[
\phi\in \calD^{\Hol}\ \Longleftarrow\ u\in \calD^{\Hol},\ v\in \calD^{\Hol}.
\]
The observation above justifies the choice of a real eigenfunction in the statement of the theorem. It is unknown whether for a reducible metric there exists any nontrivial complex-valued eigenfunction in the holomorphic extension, such that its real or imaginary part is not in the same extension. Corresponding to Theorem~\ref{thm:hol_eigenspace}, one may ask the question about the complex dimension of all such functions.

Similarly, there is a question whether an irreducible metric can have a nontrivial complex-valued eigenfunction in the holomorphic extension, which is not excluded by our theorem. By Proposition~\ref{p:hol} any such eigenfunction would produce a meromorphic vector field, and it is an interesting question whether there is any geometric implication if such a metric exists.

\vspace{5mm}

\noindent\textbf{Acknowledgement:} The authors would like to thank Rafe Mazzeo and Song Sun for many insightful discussions, and their hospitality during B.X.'s visit to Stanford University and UC Berkeley in Spring 2019. Part of the work was completed when both authors visited SCMS, Fudan University in Summer 2019.
The authors would also like to thank Pierre Albin, Zhijie Chen, Alexandre Eremenko,
Luc Hillairet,  Misha Karpukhin, Chris Kottke, Chang-Shou Lin, Gerardo Mendoza, Gabriele Mondello, Dmitri Panov,  Chin-Lung Wang, Yingyi Wu, Hao Yin and the anonymous referees for helpful suggestions and comments.
X. Z. is partially supported by the NSF grant DMS-2041823. Part of the work was supported by the NSF grant DMS-1440140 while X.Z. was in residence at MSRI in Berkeley, California, during Fall 2019 semester. B.X. is partially supported by the National Natural Science Foundation of China (Grant Nos. 11571330,11971450 and 12071449) and the Fundamental Research Funds for the Central Universities. Part of the work was done while B.X. was visiting Institute of Mathematical Sciences at ShanghaiTech University in Spring 2019.

\Addresses


\footnotesize

\begin{thebibliography}{CWWX15}

\bibitem[Bal06]{Ballmann}
Werner Ballmann.
\newblock {\em Lectures on {K}\"{a}hler manifolds}.
\newblock ESI Lectures in Mathematics and Physics. European Mathematical
  Society (EMS), Z\"{u}rich, 2006.

\bibitem[BDMM11]{BDM}
Daniele Bartolucci, Francesca De~Marchis, and Andrea Malchiodi.
\newblock Supercritical conformal metrics on surfaces with conical
  singularities.
\newblock {\em Int. Math. Res. Not. IMRN}, (24):5625--5643, 2011.

\bibitem[BS85]{BS1}
Jochen Br\"{u}ning and Robert Seeley.
\newblock Regular singular asymptotics.
\newblock {\em Adv. in Math.}, 58(2):133--148, 1985.

\bibitem[BS88]{BS2}
Jochen Br\"{u}ning and Robert Seeley.
\newblock An index theorem for first order regular singular operators.
\newblock {\em Amer. J. Math.}, 110(4):659--714, 1988.

\bibitem[Car14]{Car}
Alessandro Carlotto.
\newblock On the solvability of singular {L}iouville equations on compact
  surfaces of arbitrary genus.
\newblock {\em Trans. Amer. Math. Soc.}, 366(3):1237--1256, 2014.

\bibitem[CCW05]{CCW}
Qing Chen, Xiuxiong Chen, and Yingyi Wu.
\newblock The structure of {HCMU} metric in a {$K$}-surface.
\newblock {\em Int. Math. Res. Not.}, (16):941--958, 2005.

\bibitem[Che79]{Cheeger}
Jeff Cheeger.
\newblock On the spectral geometry of spaces with cone-like singularities.
\newblock {\em Proc. Nat. Acad. Sci. U.S.A.}, 76(5):2103--2106, 1979.

\bibitem[Che00]{Che00}
Xiuxiong Chen.
\newblock Obstruction to the existence of metric whose curvature has umbilical
  {H}essian in a {$K$}-surface.
\newblock {\em Comm. Anal. Geom.}, 8(2):267--299, 2000.

\bibitem[Che19]{ChenZ}
Zhijie Chen.
\newblock private communication, 2019.

\bibitem[CKL17]{CKL}
Zhijie Chen, Ting-Jung Kuo, and Chang-Shou Lin.
\newblock Existence and non-existence of solutions of the mean field equations
  on flat tori.
\newblock {\em Proc. Amer. Math. Soc.}, 145(9):3989--3996, 2017.

\bibitem[CLSX]{CLSX}
Q.~Chen, B.~Li, J.~Song, and B.~Xu.
\newblock {J}enkins--{S}trebel differentials and reducible cone spherical
  metrics on compact riemann surfaces.
\newblock {\em In preparation}.

\bibitem[CLW15]{CLW}
Ching-Li Chai, Chang-Shou Lin, and Chin-Lung Wang.
\newblock Mean field equations, hyperelliptic curves and modular forms: {I}.
\newblock {\em Camb. J. Math.}, 3(1-2):127--274, 2015.

\bibitem[CW11]{CW}
Qing Chen and Yingyi Wu.
\newblock Character 1-form and the existence of an {HCMU} metric.
\newblock {\em Math. Ann.}, 351(2):327--345, 2011.

\bibitem[CWWX15]{CWWX2015}
Qing Chen, Wei Wang, Yingyi Wu, and Bin Xu.
\newblock Conformal metrics with constant curvature one and finitely many
  conical singularities on compact {R}iemann surfaces.
\newblock {\em Pacific J. Math.}, 273(1):75--100, 2015.

\bibitem[CWX15]{CWX}
Qing Chen, Yingyi Wu, and Bin Xu.
\newblock On one-dimensional and singular {C}alabi's extremal metrics whose
  {G}auss curvatures have nonzero umbilical {H}essians.
\newblock {\em Israel J. Math.}, 208(1):385--412, 2015.

\bibitem[Dey18]{Dey}
Subhadip Dey.
\newblock Spherical metrics with conical singularities on 2-spheres.
\newblock {\em Geom. Dedicata}, 196:53--61, 2018.

\bibitem[EG15]{EG}
Alexandre Eremenko and Andrei Gabrielov.
\newblock On metrics of curvature 1 with four conic singularities on tori and
  on the sphere.
\newblock {\em Illinois J. Math.}, 59(4):925--947, 2015.

\bibitem[EGT16]{EGT}
A.~Eremenko, A.~Gabrielov, and V.~Tarasov.
\newblock Spherical quadrilaterals with three non-integer angles.
\newblock {\em Zh. Mat. Fiz. Anal. Geom.}, 12(2):134--167, 2016.

\bibitem[Ere04]{Ere}
A.~Eremenko.
\newblock Metrics of positive curvature with conic singularities on the sphere.
\newblock {\em Proc. Amer. Math. Soc.}, 132(11):3349--3355, 2004.

\bibitem[Ere17]{Ere2017}
Alexandre Eremenko.
\newblock Co-axial monodromy.
\newblock {\em Ann. Sc. Norm. Super. Pisa Cl. Sci.} Vol. 20 No.5(2020) 619-634.

\bibitem[Ere19]{Ere19}
Alexandre Eremenko.
\newblock On metrics of constant positive curvature with four conic
  singularities on the sphere.
\newblock {\em arXiv preprint arXiv:1905.02537}, 2019.

\bibitem[FSX17]{FSX}
Yu~Feng, Yi~Qian Shi, and Bin Xu.
\newblock On the explicit expression of a conformal metric of constant
  curvature one near a conical singularity.
\newblock {\em J. Univ. Sci. Technol. China}, 47(6):455--458, 2017.

\bibitem[GKM06]{GKM2}
Juan~B. Gil, Thomas Krainer, and Gerardo~A. Mendoza.
\newblock Resolvents of elliptic cone operators.
\newblock {\em J. Funct. Anal.}, 241(1):1--55, 2006.

\bibitem[GKM07]{GKM}
Juan~B. Gil, Thomas Krainer, and Gerardo~A. Mendoza.
\newblock Geometry and spectra of closed extensions of elliptic cone operators.
\newblock {\em Canad. J. Math.}, 59(4):742--794, 2007.

\bibitem[GM03]{GM}
Juan~B. Gil and Gerardo~A. Mendoza.
\newblock Adjoints of elliptic cone operators.
\newblock {\em Amer. J. Math.}, 125(2):357--408, 2003.

\bibitem[Her70]{Her}
Joseph Hersch.
\newblock Quatre propri\'{e}t\'{e}s isop\'{e}rim\'{e}triques de membranes
  sph\'{e}riques homog\`enes.
\newblock {\em C. R. Acad. Sci. Paris S\'{e}r. A-B}, 270:A1645--A1648, 1970.

\bibitem[Hil10]{Hill}
Luc Hillairet.
\newblock Spectral theory of translation surfaces: a short introduction.
\newblock In {\em Actes du {S}\'{e}minaire de {T}h\'{e}orie {S}pectrale et
  {G}\'{e}ometrie. {V}olume 28. {A}nn\'{e}e 2009--2010}, volume~28 of {\em
  S\'{e}min. Th\'{e}or. Spectr. G\'{e}om.}, pages 51--62. Univ. Grenoble I,
  Saint-Martin-d'H\`eres, [2010].

\bibitem[HK17]{HK}
Luc Hillairet and Alexey Kokotov.
\newblock Isospectrality, comparison formulas for determinants of {L}aplacian
  and flat metrics with non-trivial holonomy.
\newblock {\em Proc. Amer. Math. Soc.}, 145(9):3915--3928, 2017.

\bibitem[Kap17]{Ka}
Michael Kapovich.
\newblock Branched covers between spheres and polygonal inequalities in
  simplicial trees, 2017.

\bibitem[KNPP17]{KNPP}
Mikhail Karpukhin, Nikolai Nadirashvili, Alexei~V Penskoi, and Iosif
  Polterovich.
\newblock An isoperimetric inequality for laplace eigenvalues on the sphere.
\newblock {\em arXiv preprint arXiv:1706.05713}, 2017.

\bibitem[Les97]{Lesch}
Matthias Lesch.
\newblock {\em Operators of {F}uchs type, conical singularities, and asymptotic
  methods}, volume 136 of {\em Teubner-Texte zur Mathematik [Teubner Texts in
  Mathematics]}.
\newblock B. G. Teubner Verlagsgesellschaft mbH, Stuttgart, 1997.

\bibitem[Lic58]{Li}
Andr\'{e} Lichnerowicz.
\newblock {\em G\'{e}om\'{e}trie des groupes de transformations}.
\newblock Travaux et Recherches Math\'{e}matiques, III. Dunod, Paris, 1958.

\bibitem[LT92]{LuTi92}
Feng Luo and Gang Tian.
\newblock Liouville equation and spherical convex polytopes.
\newblock {\em Proc. Amer. Math. Soc.}, 116(4):1119--1129, 1992.

\bibitem[LW10]{LW}
Chang-Shou Lin and Chin-Lung Wang.
\newblock Elliptic functions, {G}reen functions and the mean field equations on
  tori.
\newblock {\em Ann. of Math. (2)}, 172(2):911--954, 2010.

\bibitem[LZ02]{LiZhu02}
Chang-Shou Lin and Xiaohua Zhu.
\newblock Explicit construction of extremal {H}ermitian metrics with finite
  conical singularities on {$S^2$}.
\newblock {\em Comm. Anal. Geom.}, 10(1):177--216, 2002.

\bibitem[Ma57]{Ma57} Y. Matsushima.
\newblock Sur la structure du groupe d'hom\'eomorphismes analytiques d'une certaine vari\'et\'e
kaehlerienne.
\newblock {\em Nagoya Math. J.}, 11:145-150, 1957.

\bibitem[McO88]{McOw}
Robert~C. McOwen.
\newblock Point singularities and conformal metrics on {R}iemann surfaces.
\newblock {\em Proc. Amer. Math. Soc.}, 103(1):222--224, 1988.

\bibitem[Mel93]{Me}
Richard~B. Melrose.
\newblock {\em The {A}tiyah-{P}atodi-{S}inger index theorem}, volume~4 of {\em
  Research Notes in Mathematics}.
\newblock A K Peters, Ltd., Wellesley, MA, 1993.

\bibitem[Moo99]{Mooers}
Edith~A. Mooers.
\newblock Heat kernel asymptotics on manifolds with conic singularities.
\newblock {\em J. Anal. Math.}, 78:1--36, 1999.

\bibitem[MP16]{MP}
Gabriele Mondello and Dmitri Panov.
\newblock Spherical metrics with conical singularities on a 2-sphere: angle
  constraints.
\newblock {\em Int. Math. Res. Not. IMRN}, (16):4937--4995, 2016.

\bibitem[MP18]{MP2}
Gabriele Mondello and Dmitri Panov.
\newblock Spherical surfaces with conical points: systole inequality and moduli
  spaces with many connected components.
{\em Geometric and Functional Analysis}, 29(4), 1110-1193.

\bibitem[MW17]{MaWe}
Rafe Mazzeo and Hartmut Weiss.
\newblock Teichm\"{u}ller theory for conic surfaces.
\newblock In {\em Geometry, analysis and probability}, volume 310 of {\em
  Progr. Math.}, pages 127--164. Birkh\"{a}user/Springer, Cham, 2017.

\bibitem[MZ17]{MaZh17}
Rafe Mazzeo and Xuwen Zhu.
\newblock Conical metrics on {R}iemann surfaces, {I}: the compactified
  configuration space and regularity.
\newblock {\em Geometry and Topology}, 24 (2020) 309–372.

\bibitem[MZ19]{MaZh19}
Rafe Mazzeo and Xuwen Zhu.
\newblock Conical metrics on {R}iemann surfaces, {II}: spherical metrics.
\newblock {\em arXiv preprint arXiv:1906.09720}, 2019.

\bibitem[Oba62]{Ob}
Morio Obata.
\newblock Certain conditions for a {R}iemannian manifold to be isometric with a
  sphere.
\newblock {\em J. Math. Soc. Japan}, 14:333--340, 1962.

\bibitem[RS75]{RS2}
Michael Reed and Barry Simon.
\newblock {\em Methods of modern mathematical physics. {II}. {F}ourier
  analysis, self-adjointness}.
\newblock Academic Press [Harcourt Brace Jovanovich, Publishers], New
  York-London, 1975.

\bibitem[RS80]{RS1}
Michael Reed and Barry Simon.
\newblock {\em Methods of modern mathematical physics. {I}}.
\newblock Academic Press, Inc. [Harcourt Brace Jovanovich, Publishers], New
  York, second edition, 1980.
\newblock Functional analysis.

\bibitem[SCLX18]{SCLX}
Jijian Song, Yiran Cheng, Bo~Li, and Bin Xu.
\newblock {Drawing Cone Spherical Metrics via Strebel Differentials}.
\newblock {\em International Mathematics Research Notices}, Volume 2020, Issue 11, June 2020, Pages 3341–3363.

\bibitem[SX20]{SongXu2020}
Jijian Song and Bin Xu.
\newblock On rational functions with more than three branch points.
\newblock {\em Algebra Colloquium}, 27(2):231-246, 2020.


\bibitem[Tro91]{Tro}
Marc Troyanov.
\newblock Prescribing curvature on compact surfaces with conical singularities.
\newblock {\em Trans. Amer. Math. Soc.}, 324(2):793--821, 1991.

\bibitem[UY00]{UY}
Masaaki Umehara and Kotaro Yamada.
\newblock Metrics of constant curvature {$1$} with three conical singularities
  on the {$2$}-sphere.
\newblock {\em Illinois J. Math.}, 44(1):72--94, 2000.

\bibitem[Zhu19a]{Zh19}
Xuwen Zhu.
\newblock Rigidity of a family of spherical conical metrics.
\newblock {\em New York J. Math.}, 26 (2020) 272--284.

\bibitem[Zhu19b]{Zh18}
Xuwen Zhu.
\newblock Spherical conic metrics and realizability of branched covers.
\newblock {\em Proc. Amer. Math. Soc.}, 147(4):1805--1815, 2019.

\end{thebibliography}
\end{document}